\documentclass{amsart}

\usepackage{graphicx} 
\usepackage{yufei}

\newcommand{\wv}[1]{\widetilde{v_{#1}}}

\title{Toroidal Hitomezashi patterns}

\author{Qiuyu Ren}

\address{Department of Mathematics, University of California, Berkeley, Berkeley, CA 94720, USA}
\email{qiuyu\_ren@berkeley.edu}

\author{Shengtong Zhang}

\address{Department of Mathematics, Stanford University, Stanford, CA 94305, USA}
\email{stzh1555@stanford.edu}

\begin{document}

\begin{abstract}
    Extending a proposal of Defant and Kravitz [Discrete Mathematics, \textbf{1}, 347 (2024)], we define Hitomezashi patterns and loops on a torus and provide several structural results for such loops. For a given pattern, our main theorems give optimal residual information regarding the Hitomezashi loop length, loop count, as well as possible homology classes of such loops. Special attention is paid to toroidal Hitomezashi patterns that are symmetric with respect to the diagonal $x = y$, where we establish a novel connection between Hitomezashi and knot theory.
\end{abstract}

\maketitle

\section{Introduction}
Hitomezashi, a type of Japanese style embroidery, has recently attracted attention due to its interesting mathematical properties. The mathematics of Hitomezashi was first studied by Pete back in 2004 \cite{Pete2008} under a different name. After Numberphile popularized the mathematical definition of Hitomezashi patterns in their YouTube video \cite{numberphile21}, Defant, Kravitz and Tenner \cite{defant2022loops, defant2023} discovered many interesting mathematical properties for loops in Hitomezashi patterns (``Hitomezashi loops"). For example, the length of any Hitomezashi loops is congruent to $4$ modulo $8$, and the area enclosed is congruent to $1$ modulo $4$. 

Thus far, the Hitomezashi patterns studied in \cite{defant2022loops,defant2023,numberphile21, Pete2008} are on the planar grid $\ZZ^2$. In this paper, we suggest a natural generalization of Hitomezashi patterns on a toroidal grid $\ZZ / M\ZZ \times \ZZ / N \ZZ$. We show that such patterns and loops enjoy some non-trivial combinatorial properties. Interestingly, our study of these properties explores a novel connection between Hitomezashi and knot theory.

Following \cite{defant2022loops}, we recall the definition of Hitomezashi patterns on $\ZZ^2$. 
\begin{definition}
    \label{defn:hitomezashi-original}
    Consider the graph $\Cloth_{\ZZ}$ on $\ZZ \times \ZZ$ with $(i, j)$ adjacent to $(i, j \pm 1)$ and $(i \pm 1, j)$. A \textbf{planar Hitomezashi pattern} is a subgraph of $\Cloth_{\ZZ}$ defined by two infinite sequences $\epsilon, \eta \in \{0, 1\}^{\ZZ}$, with edge set 
    $$\{\{(i, j), (i + 1, j): i \equiv \eta_j \bmod{2}\} \bigcup \{\{(i, j), (i, j + 1)\}: j \equiv \epsilon_i \bmod{2}\}.$$
    When we talk about this graph, we always consider the drawing of this graph on the plane with edges being straight lines.

    A \textbf{planar Hitomezashi loop} is a cycle in a planar Hitomezashi pattern. A \textbf{planar Hitomezashi path} is a path in a planar Hitomezashi pattern.
\end{definition}
This definition generalizes to $\ZZ / M\ZZ \times \ZZ / N \ZZ$ when $M$ and $N$ are both even. If either $M$ or $N$ is odd, the definition does not work since the parity of the coordinates is undefined. 

Instead, we use an approach based on Lemma 2.3 of \cite{defant2022loops}, which states that if we orient a Hitomezashi path $P$, then all edges of $P$ on the same horizontal / vertical line point in the same direction. This motivates the following definition. 
\begin{definition}
\label{defn:hitomezashi-toroidal}
Let $M, N\geq 3$ be integers. Consider the graph $\Cloth_{M, N}$ with vertex set $\ZZ / M\ZZ \times \ZZ / N\ZZ$, where a vertex $(i, j)$ is adjacent to $(i \pm 1, j)$ and $(i, j \pm 1)$. 

Let $x \in \{-1, 1\}^N, y \in \{-1, 1\}^M$ be binary strings. The \textbf{toroidal Hitomezashi pattern} given by $x, y$, denoted $\Cloth_{M, N}(x, y)$, is defined as the following orientation of $\Cloth_{M, N}$: an edge $\{(i, j), (i + 1, j)\}$ is oriented $(i, j) \to (i + 1, j)$ if $x_j = 1$, and oriented $(i + 1, j) \to (i, j)$ if $x_j = -1$. Symmetrically, an edge $\{(i, j), (i, j + 1)\}$ is oriented $(i, j) \to (i, j + 1)$ if $y_i = 1$, and oriented $(i, j + 1) \to (i, j)$ if $y_i = -1$.

A \textbf{toroidal Hitomezashi loop} is a circuit in the oriented graph $\Cloth_{M, N}(x,y)$ whose edges alternate between vertical and horizontal edges.\footnote{Note that when $M$ or $N$ is odd, a loop may pass a vertex twice. See the figure on the right in \cref{fig:flip_bit}.} 

A toroidal Hitomezashi pattern is \textbf{symmetric} if $M = N$ and $x = y$.
\end{definition}
We have code that generates toroidal Hitomezashi pattern at \cite{Ren_Toroidal_Hitomezashi_Patterns_2024}.

\begin{figure}[ht]
    \centering
    \includegraphics[width=0.45\textwidth]{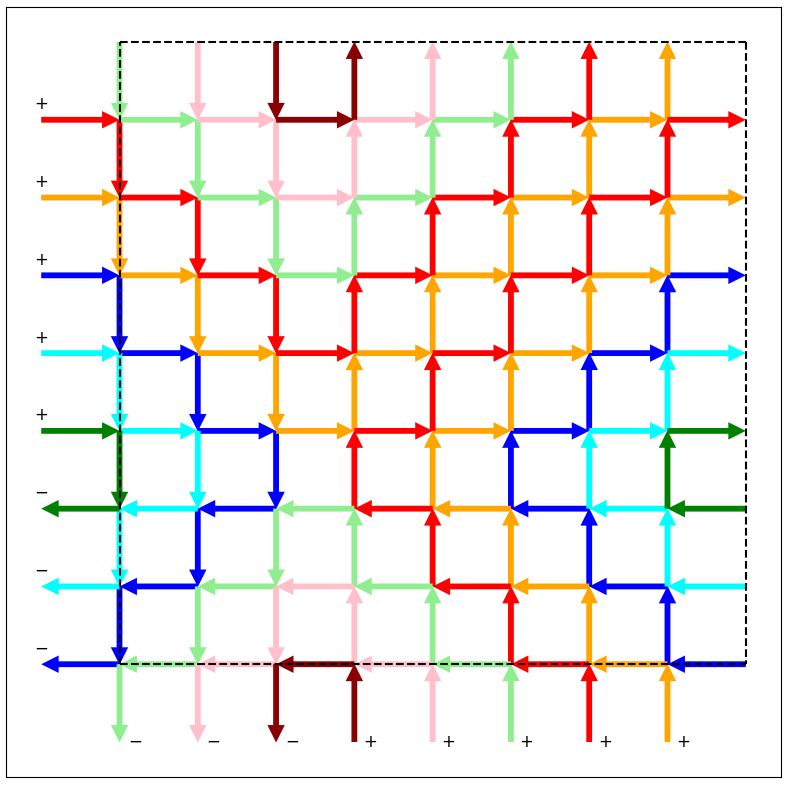}
    \caption{The symmetric toroidal Hitomezashi pattern $\Cloth_{8,8}(x,x)$ for $x=---+++++$, with Hitomezashi loops distinguished by different colors. The red and orange loops are nontrivial with homology class $(1,1)$, while the other six loops are trivial. The black dotted square represents a fundamental domain for the torus. }
    \label{fig:00011111}
\end{figure}

Let us explain the relationship between \cref{defn:hitomezashi-original} and \cref{defn:hitomezashi-toroidal}  when $M, N$ are both even. Let $C$ be a toroidal Hitomezashi pattern as in \cref{defn:hitomezashi-toroidal}. Let $A$ be the union of Hitomezashi loops in $C$ whose coordinates modulo $2$ follow the pattern $(0, 0) \to (0, 1) \to (1, 1) \to (1, 0) \to (0, 0) \to \cdots$, and let $B$ be the union of Hitomezashi loops whose coordinates modulo $2$ follow the pattern $(0, 0) \to (1, 0) \to (1, 1) \to (0, 1) \to (0, 0) \to \cdots$. When we forget the orientation, $A, B$ are Hitomezashi patterns in the sense of \cref{defn:hitomezashi-original}, and they are duals of each other in the sense of \cite[Section 7.1]{defant2022loops}. Thus, the toroidal Hitomezashi pattern in \cref{defn:hitomezashi-toroidal} decomposes into a Hitomezashi pattern in the sense of \cref{defn:hitomezashi-original} and its dual. See \cref{fig:decomposition} for an illustration.

\begin{figure}[ht]
    \centering
    \begin{minipage}{0.5\textwidth}
        \centering
        \includegraphics[width=0.9\textwidth]{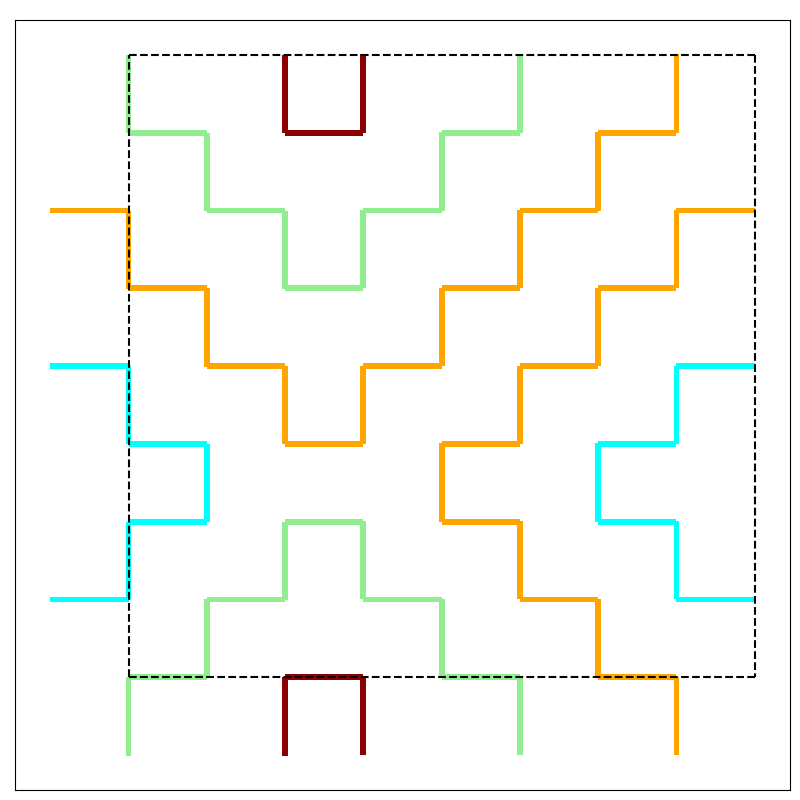} 
    \end{minipage}\hfill
    \begin{minipage}{0.5\textwidth}
        \centering
        \includegraphics[width=0.9\textwidth]{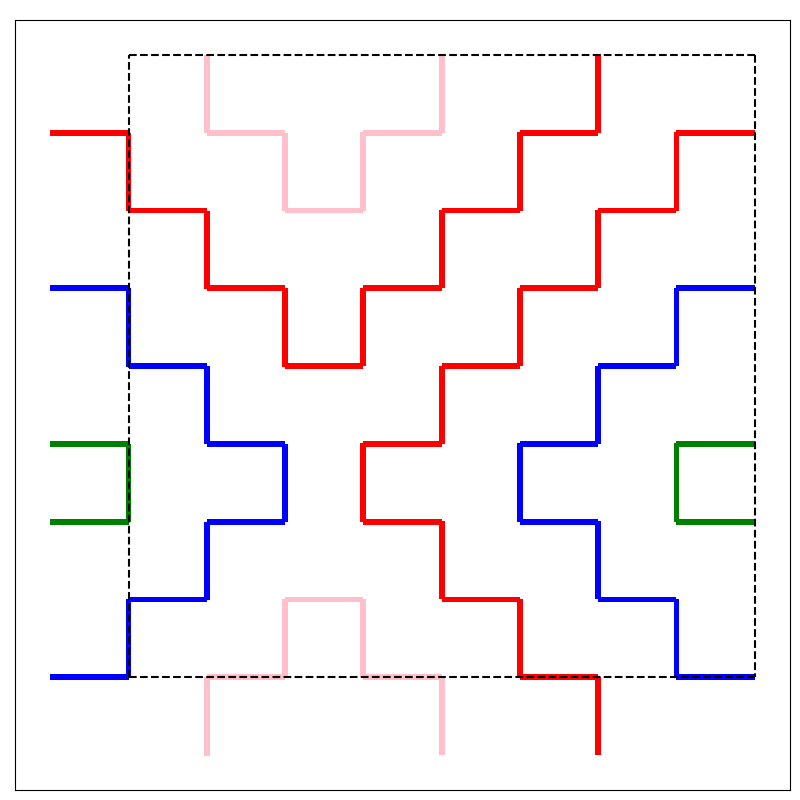} 
    \end{minipage}
    \caption{Decomposition of \cref{fig:00011111} into two Hitomezashi patterns as in \cref{defn:hitomezashi-original}.}
    \label{fig:decomposition}
\end{figure}

Perhaps the most interesting difference between planar and toroidal Hitomezashi loops is that some toroidal Hitomezashi loops are not contractible. On the torus, the isotopy class of a noncontractible simple closed curve is classified by its homology class. In our setup, the homology class of a Hitomezashi loop can be defined combinatorially. 
\begin{definition}[Combinatorial Definition of homology]
\label{defn:homology}
Given a toroidal Hitomezashi loop on $\Cloth_{M, N}(x, y)$, define its $x$-shift $\Delta x$ as the number of edges $(i, j) \to (i + 1, j)$ minus the number of edges $(i + 1, j) \to (i, j)$ on the loop, and its $y$-shift $\Delta y$ analogously. The \textbf{homology class} of a loop is defined as $(\Delta x / M, \Delta y / N)$. We call a loop \textbf{trivial} if its homology class is $(0,0)$, and \textbf{nontrivial} otherwise.
\end{definition}
We shall see that trivial toroidal Hitomezashi loops are essentially the same as planar Hitomezashi loops, while nontrivial toroidal Hitomezashi loops enjoy somewhat different properties.

We begin with some topological observations, to be used throughout the rest of the paper.

\begin{observation}\ \vspace{-8pt}
\label{obs}
\begin{enumerate}
    \item In a toroidal Hitomezashi pattern, each vertex has exactly one vertical in-edge and one horizontal in-edge. Thus, the pattern is partitioned into toroidal Hitomezashi loops, and toroidal Hitomezashi loops never crosses itself transversely.

    \item Since every toroidal Hitomezashi loop has no transversal self-intersection, and every two different toroidal Hitomezashi loops are disjoint, it is a topological fact that the homology class of any nontrivial toroidal Hitomezashi loop is $\pm(u,v)$ for some coprime integers $u,v$ independent of the chosen loop.

    \item Consider the case when the pattern is symmetric. In this case, any toroidal Hitomezashi loop cannot transversely cross the diagonal $\{(a,a)\colon a\in\RR/N\ZZ\}$ in the torus. Together with (2), we see its homology class is either $(0,0)$ or $\pm(1,1)$.
\end{enumerate}
\end{observation}

In the rest of the paper, we establish some properties concerning toroidal Hitomezashi patterns and loops. 

Let $k(x)$ denote the difference between the number of $1$'s and the number of $(-1)$'s in $x$. Let $\gcd(x, y)$ denote $\gcd(\abs{k(x)}, \abs{k(y)})$. The first property states that, when $k(x), k(y)$ are both non-zero, all nontrivial Hitomezashi loops have the same homology class, which can be expressed in terms of $k(x)$ and $k(y)$. 

\begin{theorem}[Homology class]\ \vspace{-8pt}
\label{prop:homology}
\begin{enumerate}
\item In a symmetric toroidal Hitomezashi pattern $\Cloth_{N, N}(x, x)$, every nontrivial toroidal Hitomezashi loop has homology class $(1,1)$ if $k(x) > 0$, and homology class $(-1,-1)$ if $k(x) < 0$. No nontrivial toroidal Hitomezashi loop exists if $k(x) = 0$.
\item In a general toroidal Hitomezashi pattern $\Cloth_{M, N}(x, y)$, the possible homology classes of nontrivial toroidal Hitomezashi loops are
 \begin{center}
\begin{tabular}{ |c|c|c| } 
 \hline
   & $k(x) \neq 0$ & $k(x) = 0$ \\ 
 \hline 
 $k(y) \neq 0$ & $(k(x) / \gcd(x, y), k(y) / \gcd(x, y))$ & $(0, \pm 1)$ \\ 
 \hline
 $k(y) = 0$ & $(\pm 1, 0)$ & $(0, \pm 1)$ or $(\pm 1, 0)$ \\ 
 \hline
\end{tabular}
\end{center}
\end{enumerate}
\end{theorem}
This property have two interesting consequences. First, when $k(x)$ and $k(y)$ are both nonzero, all the nontrivial Hitomezashi loops must travel in the same direction.\footnote{When $k(x) = 0$, loops with homology class $(0, 1)$ and $(0, -1)$ can coexist in the same Hitomezashi pattern. See the orange and cyan loop in the second figure of \cref{fig:loop_count}.} Second, flipping a single bit of $x$ or $y$ could have a tremendous effect on the picture of toroidal Hitomezashi loops. See \cref{fig:flip_bit} for an illustration.

\begin{figure}[ht]
    \centering
    \begin{minipage}{0.5\textwidth}
        \centering
        \includegraphics[width=0.9\textwidth]{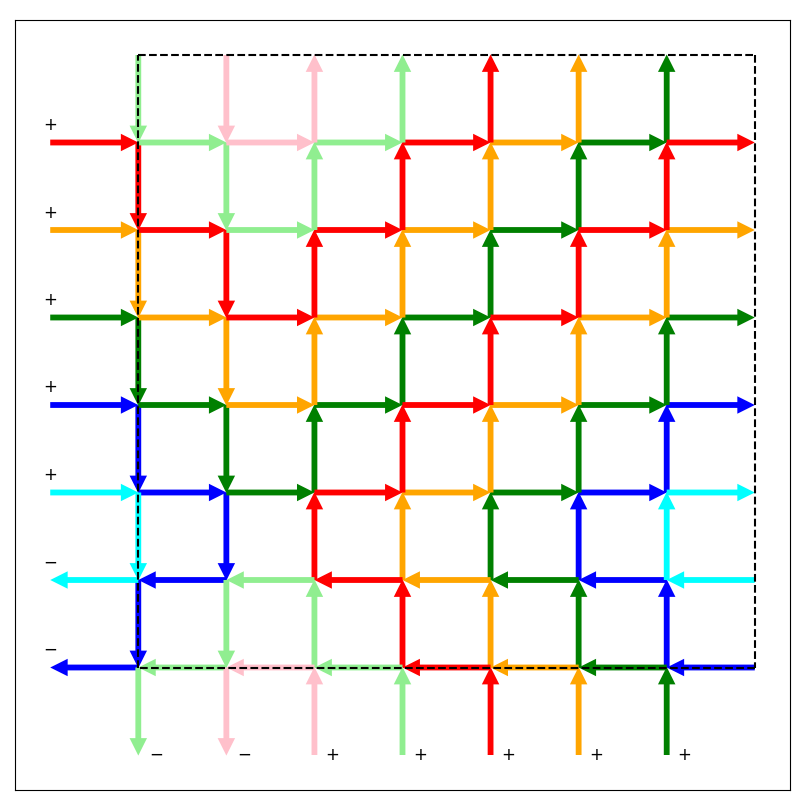} 
    \end{minipage}\hfill
    \begin{minipage}{0.5\textwidth}
        \centering
        \includegraphics[width=0.9\textwidth]{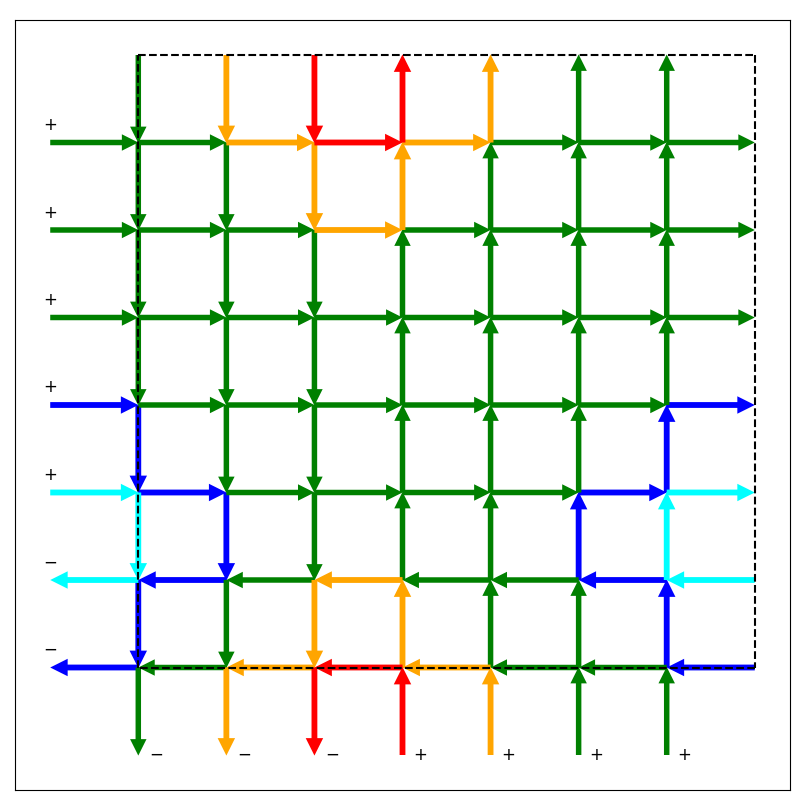} 
    \end{minipage}
    \caption{Comparison between $\Cloth_{7,7}(x, x)$ and $\Cloth_{7,7}(x, x')$, where $x = --+++++$ and $x' = ---++++$. On the right, all green edges form a single toroidal Hitomezashi loop of homology class $(3, 1)$.}
    \label{fig:flip_bit}
\end{figure}

The second property is analogous to one of the main theorems in \cite{defant2022loops}. It determines the length of toroidal Hitomezashi loops modulo $8$.

\begin{theorem}[Loop length]\ \vspace{-8pt}
\label{prop:length}
\begin{enumerate}
\item Every trivial toroidal Hitomezashi loop has length $4$ modulo $8$.
\item In a symmetric pattern $\Cloth_{N, N}(x, x)$, every nontrivial toroidal Hitomezashi loop has length $2\abs{k(x)}$ modulo 8.

\item In a general pattern $\Cloth_{M, N}(x, y)$, if a nontrivial toroidal Hitomezashi loop has homology class $(\lambda, \mu)$, then its length is congruent to $2(\mu N + \lambda M - \mu k(x))$ modulo $8$.\footnote{As $\mu k(x) = \lambda k(y)$, this expression can also be written as $2(\mu N + \lambda M) - \mu k(x) - \lambda k(y)$, so it is invariant if we switch $x$ and $y$.}
\end{enumerate}
\end{theorem}
The last property is a counting result for toroidal Hitomezashi loops on $\Cloth_{M, N}(x, y)$. 
\begin{theorem}[Loop count]
\label{prop:loop-count}\ \vspace{-8pt}
\begin{enumerate}
\item In any pattern $\Cloth_{M, N}(x, y)$, if $k(x), k(y)$ are both non-zero, then the number of nontrivial toroidal Hitomezashi loops is $\gcd(x, y)$.
\item The number of trivial toroidal Hitomezashi loops in any toroidal Hitomezashi pattern is even.
\item In a symmetric pattern $\Cloth_{N, N}(x, x)$, the total number of toroidal Hitomezashi loops is congruent to $N$ modulo $4$.
\item In a symmetric pattern $\Cloth_{N, N}(x, x)$, the minimal possible total number of toroidal Hitomezashi loops is $N$.
\end{enumerate}
\end{theorem}
We will use different approaches for each result. Let us summarize these approaches here. \cref{prop:homology} is proven using the idea of ``heights" introduced in \cite{Pete2008}. We lift the toroidal Hitomezashi loops to infinite planar Hitomezashi paths, and use height to analyze two such paths that are adjacent.

\cref{prop:length} is proven using the tool of ``excursions" as in \cite{shortproof23}. We first establish a connection between toroidal Hitomezashi loops and two types of planar Hitomezashi excursions, then use the method of induction in \cite{shortproof23} to determine the length of these excursions.

\cref{prop:loop-count}(1) is a corollary of \cref{prop:homology}. \cref{prop:loop-count}(2) can either be derived from \cref{prop:length} by considering the sum of lengths of all loops, or derived from the fact that the number of clockwise trivial Hitomezashi loops is equal to the number of counterclockwise trivial Hitomezashi loops.

This paper emerges out of an attempt to prove \cref{prop:loop-count}(3), in a knot theoretic context. The proofs of \cref{prop:loop-count}(3, 4) are thereby based on a novel connection between symmetric Hitomezashi patterns $\Cloth_{N, N}(x, x)$ and knots. Despite the simple statements, we don't see an easy combinatorial proof of these two results.

Roughly speaking, we realize $\Cloth_{N,N}(x,x)$ as the underlying graph of a link diagram of the torus link $T(N,N)$ in the sense of knot theory. In this context, the notion of Hitomezashi loops corresponds exactly to that of \textit{Seifert circles} in knot theory. \cref{prop:loop-count}(4) now follows from a deep knot-theoretic result on Seifert circles and the braid index of a link. The proof of \cref{prop:loop-count}(3) is by applying some ``triple-point moves,'' motivated by the Reidemeister III move in the knot theory, to show the loop counts for $\Cloth_{N, N}(x, x)$ and $\Cloth_{N, N}(x', x')$ are equal module $4$, where $x'$ is $x$ with two adjacent bits flipped. This allows us to reduce to the $x$ that consists of consecutive $1$'s followed by consecutive $(-1)$'s, in which case the property can be easily verified.

\begin{remark}
One can check via examples that the moduli in the above theorems are optimal. For example, the number of loops in $\Cloth_{N, N}(x, x)$ modulo any $K$ with $K\nmid4$ cannot be written as a function of $k(x)$ and $N$.
\end{remark}

The rest of the paper is structured as follows. Section \ref{sec:homology}, \ref{sec:length}, \ref{sec:knot-theory} are devoted to the proof of \cref{prop:homology}, \cref{prop:length}, \cref{prop:loop-count}, respectively. In Section \ref{sec:problems}, we present some futher directions of research.

\section*{Acknowledgement}
Shengtong Zhang is supported by the Craig Franklin Fellowship in Mathematics at Stanford University.

\section{Homology class and Height}
\label{sec:homology}
\subsection{Height}
For $x, y \in \{-1, 1\}^{\ZZ}$, let $\Cloth_{\ZZ}(x, y)$ denote the orientation of $\Cloth_{\ZZ}$ defined verbatim as \cref{defn:hitomezashi-toroidal}.

Following \cite{Pete2008}, we define the height\footnote{Note that our definition of height differs from the ``height" in \cite{Pete2008} by a factor of $2$.} for regions, edges, and Hitomezashi paths of $\Cloth_\ZZ(x,y)$. Here a region means a unit square region on the plane divided by the drawing of the graph $\Cloth_\ZZ$. 

For every oriented edge in $\Cloth_\ZZ(x,y)$, we demand the region on its left to be one unit higher than the region on its right, and the edge to be of the average height of these two regions. This uniquely defines the height on regions and edges up to an additive constant. More concretely, one may define the height of the region containing $(i+1/2,j+1/2)$ to be $\sum_{k=0}^jx_k+\sum_{k=0}^iy_k$, where $\sum_{k=0}^r$ is interpreted as $-\sum_{k=r}^{-1}$ if $r<0$.

If two regions are connected by a path on the plane that intersects the drawing of $\Cloth_\ZZ$ only discretely and does not transversely intersect any Hitomezashi path, then they have the same height. In particular, by perturbing a Hitomezashi path to its left (resp. right) except at vertices, we see all regions on the immediate left (resp. right) of a Hitomezashi path have the same height, an observation first made in \cite[Page 5-6]{Pete2008}. Therefore all edges on a Hitomezashi path have the same height, which is defined to be the height of this Hitomezashi path.

Finally, for a toroidal Hitomezashi pattern $\Cloth_{M, N}(x,y)$ with $k(x)=k(y)=0$, we may define the height of its regions, edges, and toroidal Hitomezashi loops exactly as above.

\subsection{Proof of \texorpdfstring{\cref{prop:homology}}{Theorem 1}}

(1) is a consequence of (2) together with \cref{obs}(3). Thus we only prove (2).

The sum of homology classes of all Hitomezashi loops in $\Cloth_{M, N}(x,y)$ equals to $(k(x),k(y))$. Together with \cref{obs}(2), we see if $(k(x),k(y))\ne(0,0)$, the homology class of any nontrivial toroidal Hitomezashi loop must be $$\pm(k(x)/\gcd(x,y),k(y)/\gcd(x,y)).$$ This proves (2) for the cases $k(x)=0$, $k(y)\ne0$ and $k(x)\ne0$, $k(y)=0$. Below we deal with the other two cases.

\textbf{Case 1}: $k(x),k(y)\ne0$. By symmetry we may assume $k(x)>0$.

Assume for the sake of contradiction that there exists a toroidal Hitomezashi loop $\gamma$ with homology class $-(k(x)/\gcd(x,y),k(y)/\gcd(x,y)).$
Necessarily there also exists a loop $\gamma'$ with homology class $(k(x)/\gcd(x,y),$ $k(y)/\gcd(x,y))$. The loops $\gamma,\gamma'$ cobound two annular regions on the torus. We can choose $\gamma,\gamma'$ so that $\gamma'$ is immediately below $\gamma$ in the sense that the annular region below $\gamma$ and above $\gamma'$ contains no nontrivial toroidal Hitomezashi loops (we may talk about above/below thanks to $k(x)\ne0$).

Pullback the toroidal Hitomezashi pattern $\Cloth_{M,N}(x,y)$ to the planar graph $\Cloth_\ZZ$, i.e. define infinite strings $\tilde x,\tilde y\in\{-1,1\}^\ZZ$ by repeating $x,y$ periodically, and consider the planar Hitomezashi pattern $\Cloth_\ZZ(\tilde x,\tilde y)$, which comes with a covering map onto $\Cloth_{M,N}(x,y)$. Let $\tilde\gamma$, $\tilde\gamma'$ be lifts of $\gamma$, $\gamma'$ with $\tilde\gamma'$ immediately below $\tilde\gamma$ (i.e. there are no infinite Hitomezashi paths between them). Then there is a planar path between them that does not transversely intersect any Hitomezashi path. Therefore, the regions on the left of $\tilde\gamma$ has the same height with the regions on the left of $\tilde\gamma'$. Hence $\tilde\gamma$ and $\tilde\gamma'$ have the same height.

Since $k(x)>0$, by cyclically permuting $x$ we may assume all nonempty partial sums of $x$ are positive, i.e. $\sum_{k=0}^j\tilde x_k>0$ for all $j\ge0$. Since the second component of the homology class of $\gamma'$ is nonzero, $\tilde\gamma'$ contains an edge $e_1$ of the form $(i,0)\to(i+1,0)$. Since the first component of the homology class of $\gamma$ is negative and $\tilde\gamma$ is above $\tilde\gamma'$, $\tilde\gamma$ contains an edge $e_2$ of the form $(i+1,j)\to(i,j)$ with $j>0$. Now, the region on the right to $e_2$ and the region on the right to $e_1$ have the same height because $\tilde\gamma,\tilde\gamma'$ have the same height. On the other hand, their height difference equals to $\sum_{k=0}^j\tilde x_k>0$. This is a contradiction.

\textbf{Case 2}: $k(x)=k(y)=0$.

By cyclically permuting $x,y$, we may assume all partial sums of $x$ are nonpositive and all partial sums of $y$ are nonnegative. Then the height of $(-1/2,-1/2)$ is maximal among all $(-1/2,j-1/2)$ and minimal among all $(i-1/2,-1/2)$. Consequently, all edges of the form $\{(-1,j),(0,j)\}$ have strictly lower height than all edges of the form $\{(i,-1),(i,0)\}$. It follows that for a fixed toroidal Hitomezashi loop, it either passes no edge of the former type, or passes no edge of the latter type. This forbids the existence of a Hitomezashi loop with homology class $(u,v)$, $uv\ne0$.\qed

\section{Loop length and Excursion}
\label{sec:length}
In this section, we prove Theorem \ref{prop:length} based on the excursion arguments in \cite{shortproof23}.

\subsection{Excursions} 
Let $\gamma$ be a nontrivial Hitomezashi loop on $\Cloth_{M, N}(x, y)$, and let $\tgamma$ be its lift to $\Cloth_{\ZZ}(\Tilde{x}, \Tilde{y})$, where $\Tilde{x}, \Tilde{y} \in \{\pm 1\}^{\ZZ}$ are periodic extensions of $x, y$. Note that $\tgamma$ is an infinite planar Hitomezashi path. 

The goal of this subsection is to introduce ``Hitomezashi excursions", and chop $\tgamma$ up into these excursions.\footnote{Note that \cite{Pete2008} introduces a different definition of ``excursion".}
\begin{definition}
Let $a,b$ be integers. 

An \textbf{(Hitomezashi) $a$-excursion} is a planar Hitomezashi path with at least three vertices, start vertex $(a-1, i)$ for some $i \in \ZZ$, end vertex $(a-1, j)$ for some $j > i$, and all other vertices lying in the half plane $\cH_a = \{(x, y): x \geq a\}$.

An \textbf{(Hitomezashi) $(a, b)$-excursion} is a planar Hitomezashi path with start vertex $(p, b - 1)$ for some $p \geq a$, end vertex $(a - 1, q)$ for some $q \geq b$, and all other vertices lying in the quadrant $\cH_{a, b} = \{(x, y): x \geq a, y \geq b\}$.
\end{definition}
Aside from the restriction $j > i$, our definition of $a$-excursion is the same as in \cite{shortproof23}. Note that the first and last edge in an $a$-excursion must be horizontal. In an $(a, b)$-excursion, the first edge is vertical and the last edge is horizontal.

The next lemma explains how $\tgamma$ can be chopped up into Hitomezashi excursions.
\begin{lemma}
    \label{lem:excursion-in-lifting}
Let $L$ be the length of $\gamma$. Let $(\lambda, \mu)$ be the homology class of $\gamma$. Denote the vertices of $\tgamma$ by $\cdots, \widetilde{v_{-1}}, \widetilde{v_0}, \widetilde{v_1}, \cdots$.
\begin{enumerate}
    \item If $\lambda < 0, \mu > 0$, there exists $i < 0, j > 0$ such that for any $k, \ell \geq 0$, the subpath of $\tgamma$ formed by $\{\widetilde{v_t}: t \in [i - kL, j + \ell L]\}$ is an $(a, b)$-excursion for some $a, b \in \ZZ$.

    \item If $\lambda = 0, \mu > 0$, there exists an $a \in \ZZ$ such that $\tgamma$ can be partitioned into a disjoint union of $a$-excursions and upward-pointing edges lying on the line $x = a - 1$.
\end{enumerate}
\end{lemma}
\begin{proof}
Let the coordinates of $\widetilde{v_i}$ be $(a_i, b_i)$. 

(1) Assume $\lambda < 0, \mu > 0$. As $b_{m + L} = b_m + \mu N$ for any $m \in \ZZ$, the quantity $\overline{b}_m = \inf_{n \geq m} b_n$ is finite and satisfies $\lim_{a \to -\infty} \overline{b}_a = -\infty$. Thus, there exists some $i < 0$ such that $\overline{b}_i < \overline{b}_{i + 1}$. For this $i$, we have $b_i < \inf_{n \geq i + 1} b_n$. Similarly, there exists some $j > 0$ such that $a_j < \inf_{n \leq j - 1} a_n$.

Take $a = a_j + 1$ and $b = b_i + 1$. By definition, for any $i + 1 \leq n \leq j - 1$, we have $a_n > a_j = a - 1$ and $b_n > b_i = b - 1$, so $\widetilde{v_i}$ lies in $\cH_{a, b}$. Furthermore, $b_i > b_j = b - 1$ and $a_j > a_i = a - 1$. So $\{v_t: t \in [i, j]\}$ form an $(a, b)$-excursion.

For any $k, \ell \geq 0$, we have $b_{m - k L} = b_m - k \mu N$ for any $m \in \ZZ$, so $b_{i - kL} < \inf_{n \geq i - kL + 1} b_n$. Similarly, we have $a_{j + \ell L} < \inf_{n \leq j + \ell L - 1} a_n$. So by the same reasoning as the previous paragraph  $\{v_t: t \in [i - kL, j + \ell L]\}$ is an $(a, b)$-excursion for $a = a_{j + \ell L } + 1, b = b_{i - kL} + 1$.

(2) Assume $\lambda = 0, \mu > 0$. As $a_{m + L} = a_m + \mu N = a_m$, the quantity $a = 1 + \inf_{m \in \ZZ} a_m$ is well-defined. Removing all edges on $x = a - 1$ partitions $\tgamma$ into finite Hitomezashi paths.

We first show that the edges of $\tgamma$ lying on the line $x = a - 1$ point upwards. Let $(\wv{i}, \wv{i + 1})$ be such an edge. The subpaths $\{\wv{n}: i - L + 1\leq n \leq i\}$ and $\{\wv{n}: i + 1\leq n \leq i + L\}$ are disjoint paths in $\cH_{a - 1}$, with endpoints $\{(a - 1, b_{i + 1} - \mu N), (a - 1, b_{i})\}$ and $\{(a - 1, b_{i + 1}), (a - 1, b_{i} + \mu N)\}$ respectively. If $b_i > b_{i + 1}$, then
$$b_{i + 1} - \mu N < b_{i + 1} < b_i < b_{i} + \mu N$$
contradicting Lemma 2.3 of \cite{shortproof23}. Thus the edge $(\wv{i}, \wv{i + 1})$ must point upward.

Let $(\wv{i}, \wv{i + 1})$ and $(\wv{j}, \wv{j + 1})$ be consecutive edges of $\tgamma$ lying on $x = a - 1$. The subpaths $\{\wv{n}: i + 1 \leq n \leq j\}$ and $\{\wv{n}: j + 1\leq n \leq i + L + 1\}$ are disjoint paths in $\cH_{a - 1}$, with endpoints $\{(a - 1, b_{i + 1}), (a - 1, b_{j})\}$ and $\{(a - 1, b_{j + 1}), (a - 1, b_{i + 1} + \mu N)\}$. If $b_{j + 1} < b_{i + 1}$, then we have
$$b_{j} < b_{j + 1} < b_{i + 1} < b_{i + 1} + \mu N$$
contradicting Lemma 2.3 of \cite{shortproof23}. Thus $b_{j + 1} > b_{i + 1}$, which implies $b_j > b_{i + 1}$ since $b_j \neq b_{i + 1}$.

Therefore, the segment of $\tgamma$ between consecutive edges on $x = a - 1$ form an $a$-excursion.
\end{proof}

\subsection{Length of excursions}
To prove our claim on the length of toroidal Hitomezashi loops, we prove congruences for the lengths of excursions. Two claims we need have already been established.
\begin{lemma}[\cite{defant2022loops}]
    \label{thm:loop-length}
    The length of a planar Hitomezashi loop is congruent to $4$ modulo $8$.
\end{lemma}
\begin{lemma}[\cite{shortproof23}]
    \label{thm:half-excursion-length}
    The length of an $a$-excursion from $(a - 1, i)$ to $(a - 1, j)$ is congruent to $2(j - i) + 1$ modulo $8$.    
\end{lemma}

Here we prove an additional result on the length of $(a,b)$-excursions.
\begin{lemma}
    \label{lem:quarter-excursion-length}
    On $\Cloth_{\ZZ}(x, y)$, let $\cC$ be an $(a, b)$-excursion from $(p, b - 1)$ to $(a - 1, q)$. Let $k(\cC)$ denote 
    $$k(\cC) = -\sum_{k = b}^q x_k.$$
    Then the length of $\cC$ is congruent to $2(q - b - p + a + k(\cC))$ modulo $8$.
\end{lemma}
\begin{proof}

We induct on the length of $\cC$. When the length of $\cC$ is $2$, the result is trivial.

For the induction step, we proceed analogous to \cite{shortproof23}. Let $(p_1, b), (p_2, b), \cdots, (p_t, b)$ be the starting vertex of every edge of $\cC$ that lies on the horizontal line $y = b$, in the order they appear in $\cC$. Clearly, $p_1 = p$. Using a similar argument as in \cite{shortproof23}, all $p_i$ have the same parity. Furthermore, one of two cases must hold.

\textbf{Case 1}: $x_b = -1$, and $p_1 > p_2 > \cdots > p_t$. Then we can partition $\cC$ into the following subpaths: the edge $(p, b - 1) \to (p, b)$, the subpaths $\cC_i$ from $(p_i - 1, b)$ to $(p_{i + 1}, b)$ for each $1 \leq i \leq t - 1$, the edges $(p_i, b) \to (p_i - 1, b)$ for each $1 \leq i \leq t$, and the subpath $\cC'$ from $(p_t - 1, b)$ to $(a - 1, q)$. Each $\cC_i$ is a rotated $b$-excursion, and $\cC'$ is an $(a, b + 1)$-excursion, so we can apply \cref{thm:half-excursion-length} and the induction hypothesis to conclude that
\begin{align*}
    \abs{\cC} &= 2 + \sum_{i = 1}^{t - 1} (\abs{\cC_i} + 1) + \abs{\cC'} \\
    &\equiv 2 + \sum_{i = 1}^{t - 1} 2(p_i - p_{i + 1}) + 2(q - (b + 1) - (p_t - 1) + a + k(\cC')) \\
    &\equiv 2 + 2(p - p_t) + 2(q - (b + 1) - (p_t - 1) + a + (k(\cC) - 1)) \\
    &\equiv 2(q - b - p + a + k(\cC)) \pmod{8}.
\end{align*}
\textbf{Case 2}: $x_b = 1$, and $p_1 < p_2 < \cdots < p_t$. Then we can partition $\cC$ into the following subpaths: the edge $(p, b - 1) \to (p, b)$, the segments $\cC_i$ from $(p_i + 1, b)$ to $(p_{i + 1}, b)$ for each $1 \leq i \leq t - 1$, the edges $(p_i, b) \to (p_i + 1, b)$ for each $1 \leq i \leq t$, and the segments $\cC'$ from $(p_t + 1, b)$ to $(a - 1, q)$. Each $\cC_i$ is a rotated $b$-excursion, and $\cC'$ is an $(a, b + 1)$-excursion, so we can apply \cref{thm:half-excursion-length} and the induction hypothesis to conclude that
\begin{align*}
    \abs{\cC} &= 2 + \sum_{i = 1}^{t - 1} (\abs{\cC_i} + 1) + \abs{\cC'} \\
    &\equiv 2 + \sum_{i = 1}^{t - 1} 2(p_{i + 1} - p_i) + 2(q - (b + 1) - (p_t + 1) + a + k(\cC')) \\
    &\equiv 2 + 2(p_t - p) + 2(q - (b + 1) - (p_t + 1) + a + (k(\cC) + 1)) \\
    &\equiv 2(q - b - p + a + k(\cC)) \pmod{8}.
\end{align*}
In either cases the induction hypothesis holds. So the induction is complete.
\end{proof}
We can finally prove \cref{prop:length}.

\begin{proof}[Proof of \cref{prop:length}]
(1) If $\gamma$ is a trivial toroidal Hitomezashi loop, then its lifting $\tgamma$ is a planar Hitomezashi loop of the same length, which is congruent to $4$ modulo $8$ by \cref{thm:loop-length}.

(2) This follows from (3) and \cref{prop:homology}(1).

(3) First note that if $\gamma$ is a nontrivial toroidal Hitomezashi loop in $\Cloth_{M, N}(x, y)$, then the reflection $\gamma^{\circ}$ about the $x$-axis is a nontrivial toroidal Hitomezashi loop in $\Cloth_{M, N}(x^{\circ}, -y)$, where $x^{\circ}_k = x_{-k}.$ If $\gamma$ has homology 
class $(\lambda, \mu)$, then $\gamma^{\circ}$ has homology class $(\lambda, -\mu)$. As $k(x)$ and $N$ have the same parity, we have
$$2(\mu N + \lambda M) - 2 \mu k(x) \equiv 
2(-\mu N + \lambda M) + 2 \mu k(x) \pmod{8}.$$
Hence, if the result holds for $\gamma$, it holds for its reflection about the $x$-axis. By the symmetry observed earlier, the same is true for reflection about the $y$-axis and the diagonal $y = x$. Hence, it suffices to prove the result in the cases $\lambda < 0, \mu > 0$ and $\lambda = 0, \mu > 0.$

First assume $\lambda < 0,  \mu > 0$. By \cref{lem:excursion-in-lifting}, there exists $i < 0, j > 0$ such that for any $k, \ell \geq 0$, the subpath of $\tgamma$ formed by $\{\wv{t}: t \in [i - kL, j + \ell L]\}$ is an $(a, b)$-excursion for some $a, b \in \ZZ$. Let $\wv{i} = (p, b - 1)$ and $\wv{j} = (a - 1, q)$. By \cref{lem:quarter-excursion-length} applied to the subpath of $\tgamma$ from $\wv{i}$ to $\wv{j}$, we have
$$j - i \equiv 2(q - b - p + a) - 2 \sum_{k = b}^q \Tilde{x}_k \pmod{8}.$$
We also have $\wv{j + L} = (a + \lambda M - 1, q + \mu N )$. By \cref{lem:quarter-excursion-length} applied to the subpath of $\tgamma$ from $\wv{i}$ to $\wv{j + L}$, we have
$$j - i + L \equiv 2(q - b - p + a + \lambda M + \mu N) - 2 \sum_{k = b}^{q + \mu N} \Tilde{x}_k \pmod{8}.$$
Subtracting the two equations, we conclude that
$$L \equiv 2(\lambda M + \mu N) - 2 \mu k(x) \pmod{8} $$
where we used the fact that $k(x) < 0$ from \cref{prop:homology}.

In the case $\lambda = 0, \mu > 0$, note that $k(x) = 0$ by \cref{prop:homology}. Let $\tgamma$ be the lifting of $\gamma$. By \cref{lem:excursion-in-lifting},  there exists an $a \in \ZZ$ such that $\tgamma$ can be partitioned into a disjoint union of $a$-excursions and upward-pointing edges lying on the line $x = a - 1$. Let $(a - 1, b_0), (a - 1, b_1), \cdots$ be the starting points of the upward-pointing edges lying on the line $x = a - 1$. There must exist some $j$ such that $b_j = b_0 + \mu N.$ For any $i$, the subpath of $\tgamma_i$ from $(a - 1, b_i + 1)$ to $(a - 1, b_{i + 1})$ is an $a$-excursion, so \cref{thm:half-excursion-length} gives
$$\abs{\tgamma_i} \equiv 2b_{i + 1} - 2b_i - 1 \bmod{8}.$$
Summing for $i = 0, 1, \cdots, j - 1$, we get
$$L = \sum_{i = 0}^{j - 1} (\abs{\tgamma_i} + 1) \equiv 2(b_j - b_0) \equiv 2 \mu N \bmod{8}.$$
This completes the proof of \cref{prop:length}.
\end{proof}

\section{Loop count; Symmetric Hitomezashi patterns and knot theory}
\label{sec:knot-theory}
In this section we give a proof to \cref{prop:loop-count} motivated by knot theory.

Since the total homology class of all toroidal Hitomezashi loops on $\Cloth_{M,N}(x,y)$ is equal to $(k(x),k(y))$, (1) is immediate from \cref{prop:homology}(2). 

We prove (2) as follows. At each vertex, each loop turns $90^{\circ}$ clockwise or counterclockwise. By elementary geometry, the quantity 
$$(\#\text{clockwise turn} - \#\text{counterclockwise turn})$$ equals $4$ for clockwise trivial Hitomezashi loops, $-4$ for counterclockwise trivial Hitomezashi loops, and $0$ for nontrivial Hitomezashi loops. On the other hand, the total number of clockwise turns across all loops is equal to the total number of counterclockwise turns, as each vertex in $\Cloth_{M, N}(x, y)$ contributes one clockwise and one counterclockwise turn. So the number of clockwise trivial Hitomezashi loops is equal to the number of counterclockwise trivial Hitomezashi loops in $\Cloth_{M, N}(x, y)$, which implies the desired statement.

For (3), it suffices to perform the proof in the following two steps.
\begin{lemma}\label{lem:two_blocks}
\cref{prop:loop-count}(3) holds if $x$ only has at most two blocks, i.e. (up to cyclic permutation) $x$ is some number of $1$'s followed by some number of $-1$'s. In fact, the loop count is exactly $N$ in this case.
\end{lemma}
\begin{lemma}\label{lem:switch}
If \cref{prop:loop-count}(3) holds for $x$, then it also holds for any $x'$ obtained by switching two adjacent bits of $x$.
\end{lemma}

\cref{lem:two_blocks} is a direct verification. Instead of carrying out the details, we refer the readers to \cref{fig:00011111} and the figure on the left in \cref{fig:flip_bit}.

For (4), after \cref{lem:two_blocks}, it suffices to show any the loop count is at least $N$ for any pattern $x\in\{-1,1\}^N$.

We devote the next subsection to generalize the notion of Hitomezashi loops to a wider class of graphs, in reminiscence of knot theory. After rephrasing our loop counting question, (4) comes for free, provided one assumes some deep results in knot theory. After that, we turn to proving \cref{lem:switch}, which will establish (3).

\subsection{Hitomezashi loops and Seifert circles}\label{sec:annulus}
In our setup, since any toroidal Hitomezashi loop does not cross the diagonal, we can redraw the grid onto the annulus as in \cref{fig:10110_move}, where the open segments on the left and right are pairwisely identified to close the grid up. Here, we think of the grid as the union of $N$ closed strands, each consists of one horizontal and one vertical side both of length $N$. Then the orientation given by the binary string $x$ corresponds to an orientation to each of the strands. The central circle of the annulus has homology class $\pm(1,1)$.

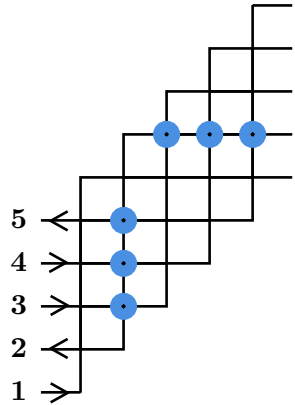
\begin{figure}[ht]
    \centering
\tikzset{every picture/.style={line width=0.75pt}} 

\begin{tikzpicture}[x=0.75pt,y=0.75pt,yscale=-.7,xscale=.7]

\draw  [draw opacity=0][line width=1]  (332.68,85.29) -- (363.67,85.29) -- (363.67,54.29) -- (332.68,54.29) -- cycle ; \draw  [line width=1]   ; \draw  [line width=1]   ; \draw  [line width=1]  (332.68,85.29) -- (363.67,85.29) -- (363.67,54.29) -- (332.68,54.29) -- cycle ;
\draw  [draw opacity=0][line width=1]  (239.68,209.28) -- (332.68,209.28) -- (332.68,178.28) -- (239.68,178.28) -- cycle ; \draw  [line width=1]  (270.68,209.28) -- (270.68,178.28)(301.68,209.28) -- (301.68,178.28) ; \draw  [line width=1]   ; \draw  [line width=1]  (239.68,209.28) -- (332.68,209.28) -- (332.68,178.28) -- (239.68,178.28) -- cycle ;
\draw  [draw opacity=0][line width=1]  (270.68,147.28) -- (363.67,147.28) -- (363.67,116.28) -- (270.68,116.28) -- cycle ; \draw  [line width=1]  (301.68,147.28) -- (301.68,116.28)(332.68,147.28) -- (332.68,116.28) ; \draw  [line width=1]   ; \draw  [line width=1]  (270.68,147.28) -- (363.67,147.28) -- (363.67,116.28) -- (270.68,116.28) -- cycle ;
\draw  [draw opacity=0][line width=1]  (239.68,178.28) -- (363.67,178.28) -- (363.67,147.28) -- (239.68,147.28) -- cycle ; \draw  [line width=1]  (270.68,178.28) -- (270.68,147.28)(301.68,178.28) -- (301.68,147.28)(332.68,178.28) -- (332.68,147.28) ; \draw  [line width=1]   ; \draw  [line width=1]  (239.68,178.28) -- (363.67,178.28) -- (363.67,147.28) -- (239.68,147.28) -- cycle ;
\draw  [draw opacity=0][line width=1]  (301.68,116.28) -- (363.67,116.28) -- (363.67,85.29) -- (301.68,85.29) -- cycle ; \draw  [line width=1]  (332.68,116.28) -- (332.68,85.29) ; \draw  [line width=1]   ; \draw  [line width=1]  (301.68,116.28) -- (363.67,116.28) -- (363.67,85.29) -- (301.68,85.29) -- cycle ;
\draw  [draw opacity=0][line width=1]  (239.68,240.27) -- (301.68,240.27) -- (301.68,209.28) -- (239.68,209.28) -- cycle ; \draw  [line width=1]  (270.68,240.27) -- (270.68,209.28) ; \draw  [line width=1]   ; \draw  [line width=1]  (239.68,240.27) -- (301.68,240.27) -- (301.68,209.28) -- (239.68,209.28) -- cycle ;
\draw  [draw opacity=0][line width=1]  (239.68,271.27) -- (270.68,271.27) -- (270.68,240.27) -- (239.68,240.27) -- cycle ; \draw  [line width=1]   ; \draw  [line width=1]   ; \draw  [line width=1]  (239.68,271.27) -- (270.68,271.27) -- (270.68,240.27) -- (239.68,240.27) -- cycle ;
\draw  [draw opacity=0][line width=1]  (363.67,147.28) -- (392.22,147.28) -- (392.22,23) -- (363.67,23) -- cycle ; \draw  [line width=1]  (363.67,147.28) -- (363.67,23) ; \draw  [line width=1]  (363.67,147.28) -- (392.22,147.28)(363.67,116.28) -- (392.22,116.28)(363.67,85.29) -- (392.22,85.29)(363.67,54.29) -- (392.22,54.29)(363.67,23.29) -- (392.22,23.29) ; \draw  [line width=1]   ;
\draw  [draw opacity=0][line width=1]  (239.68,178.28) -- (211.14,178.28) -- (211.14,302.56) -- (239.68,302.56) -- cycle ; \draw  [line width=1]  (239.68,178.28) -- (239.68,302.56) ; \draw  [line width=1]  (239.68,178.28) -- (211.14,178.28)(239.68,209.28) -- (211.14,209.28)(239.68,240.27) -- (211.14,240.27)(239.68,271.27) -- (211.14,271.27)(239.68,302.27) -- (211.14,302.27) ; \draw  [line width=1]   ;
\draw  [line width=1]  (218.16,309.5) -- (230.39,302.64) -- (218.16,295.78) ;
\draw  [line width=1]  (218.16,246.51) -- (230.39,239.65) -- (218.16,232.78) ;
\draw  [line width=1]  (217.17,215.51) -- (229.41,208.65) -- (217.17,201.79) ;
\draw  [line width=1]  (230.39,171.78) -- (218.16,178.65) -- (230.39,185.51) ;
\draw  [line width=1]  (230.39,263.78) -- (218.16,270.65) -- (230.39,277.51) ;
\draw  [color={rgb, 255:red, 74; green, 144; blue, 226 }  ,draw opacity=1 ][line width=6]  (266.68,240.27) .. controls (266.68,238.07) and (268.47,236.27) .. (270.68,236.27) .. controls (272.89,236.27) and (274.68,238.07) .. (274.68,240.27) .. controls (274.68,242.48) and (272.89,244.27) .. (270.68,244.27) .. controls (268.47,244.27) and (266.68,242.48) .. (266.68,240.27) -- cycle ;
\draw  [color={rgb, 255:red, 74; green, 144; blue, 226 }  ,draw opacity=1 ][line width=6]  (266.68,209.28) .. controls (266.68,207.07) and (268.47,205.28) .. (270.68,205.28) .. controls (272.89,205.28) and (274.68,207.07) .. (274.68,209.28) .. controls (274.68,211.49) and (272.89,213.28) .. (270.68,213.28) .. controls (268.47,213.28) and (266.68,211.49) .. (266.68,209.28) -- cycle ;
\draw  [color={rgb, 255:red, 74; green, 144; blue, 226 }  ,draw opacity=1 ][line width=6]  (266.68,178.28) .. controls (266.68,176.07) and (268.47,174.28) .. (270.68,174.28) .. controls (272.89,174.28) and (274.68,176.07) .. (274.68,178.28) .. controls (274.68,180.49) and (272.89,182.28) .. (270.68,182.28) .. controls (268.47,182.28) and (266.68,180.49) .. (266.68,178.28) -- cycle ;
\draw  [color={rgb, 255:red, 74; green, 144; blue, 226 }  ,draw opacity=1 ][line width=6]  (297.68,116.28) .. controls (297.68,114.07) and (299.47,112.28) .. (301.68,112.28) .. controls (303.89,112.28) and (305.68,114.07) .. (305.68,116.28) .. controls (305.68,118.49) and (303.89,120.28) .. (301.68,120.28) .. controls (299.47,120.28) and (297.68,118.49) .. (297.68,116.28) -- cycle ;
\draw  [color={rgb, 255:red, 74; green, 144; blue, 226 }  ,draw opacity=1 ][line width=6]  (328.68,116.28) .. controls (328.68,114.07) and (330.47,112.28) .. (332.68,112.28) .. controls (334.88,112.28) and (336.68,114.07) .. (336.68,116.28) .. controls (336.68,118.49) and (334.88,120.28) .. (332.68,120.28) .. controls (330.47,120.28) and (328.68,118.49) .. (328.68,116.28) -- cycle ;
\draw  [color={rgb, 255:red, 74; green, 144; blue, 226 }  ,draw opacity=1 ][line width=6]  (359.67,116.28) .. controls (359.67,114.07) and (361.46,112.28) .. (363.67,112.28) .. controls (365.88,112.28) and (367.67,114.07) .. (367.67,116.28) .. controls (367.67,118.49) and (365.88,120.28) .. (363.67,120.28) .. controls (361.46,120.28) and (359.67,118.49) .. (359.67,116.28) -- cycle ;

\draw (187,292) node [anchor=north west][inner sep=0.75pt]   [align=left] {\textbf{{\large 1}}};
\draw (187,261) node [anchor=north west][inner sep=0.75pt]   [align=left] {\textbf{{\large 2}}};
\draw (187,230) node [anchor=north west][inner sep=0.75pt]   [align=left] {{\large \textbf{3}}};
\draw (187,199) node [anchor=north west][inner sep=0.75pt]   [align=left] {\textbf{{\large 4}}};
\draw (187,168) node [anchor=north west][inner sep=0.75pt]   [align=left] {\textbf{{\large 5}}};

\end{tikzpicture}

    \caption{The symmetric toroidal Hitomezashi pattern for $x=+-++-$, redrawn on the annulus. The blue points indicate the intersections that the first strand passes across in the sequence of triple point moves that exchanges the first and second strands.}
    \label{fig:10110_move}
\end{figure}
\begin{figure}
    \centering
    \scalebox{.6}{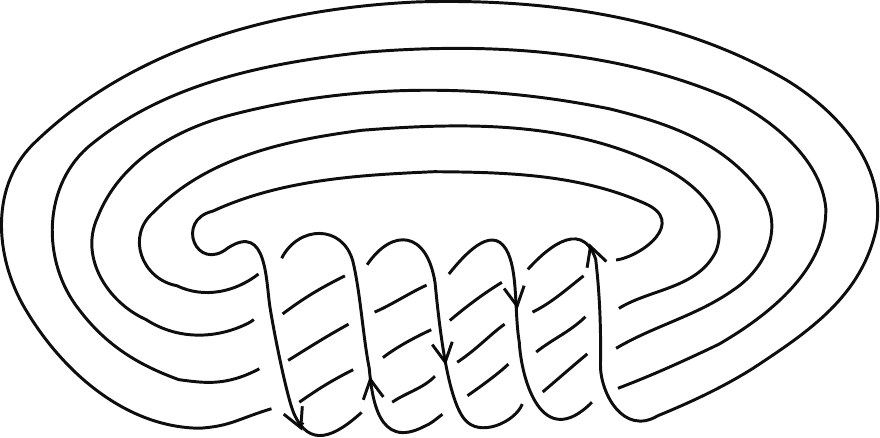}
    \caption{A link diagram for the torus link $T(5,5)$, equipped with an orientation. The underlying link-like graph is $G'(x)$, $x=+-++-$.}
    \label{fig:T55}
\end{figure}

We generalize the notion of Hitomezashi loops a bit. In a $4$-regular graph drawn on the plane, any two edges sharing the same vertex are either adjacent or opposite to each other. A \textbf{link-like graph} is a directed $4$-regular planar graph, drawn on the plane such that any two opposite edges at any vertex point to the same direction in the plane (See \cref{fig:crossing}). Such a graph is an oriented link diagram (in the sense of knot theory) with the overpass/underpass information forgotten. A \textbf{Hitomezashi loop} of a link-like graph $G$ is a directed circuit in $G$ whose any two consecutive edges are adjacent. Then, our grid on the annulus, upon forgetting the $2N$ corners, is such a graph $G'(x)$ (it is planar because the annulus embeds naturally into the plane), and its Hitomezashi loops correspond exactly to those in the sense of our previous definition.

\begin{figure}[t]

\tikzset{every picture/.style={line width=0.75pt}} 

\begin{tikzpicture}[x=0.75pt,y=0.75pt,yscale=-.8,xscale=.8]
\centering

\draw [line width=1]    (271.21,95.5) .. controls (323.83,135.15) and (256.18,195.41) .. (280.23,252.5) ;
\draw [line width=1]    (178,180.14) .. controls (201,194.5) and (299.78,134.15) .. (396,195.2) ;
\draw  [line width=3]  (289,169) .. controls (289,169.38) and (287.88,170.5) .. (286.5,170.5) .. controls (285.12,170.5) and (284,169.38) .. (284,168) .. controls (284,166.62) and (285.12,165.5) .. (286.5,165.5) .. controls (287.88,165.5) and (289,166.62) .. (289,168) -- cycle ;
\draw  [line width=1]  (300.27,117.22) -- (292.72,135.2) -- (281.32,119.38) ;
\draw  [line width=1]  (287.93,202.55) -- (274.89,217.05) -- (269.32,198.36) ;
\draw  [line width=1]  (345.06,182.86) -- (330.44,169.96) -- (349.08,164.22) ;
\draw  [line width=1]  (246.73,182.74) -- (228.37,176.17) -- (243.56,163.94) ;

\end{tikzpicture}

\caption{A possible orientation assignments to edges adjacent to a vertex in a link-like graph.}
\label{fig:crossing}
\end{figure}
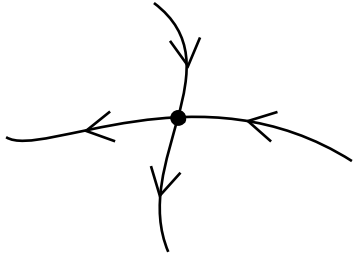

Two remarks are in order. First, in the knot theoretic context, our definition of Hitomezashi loops corresponds to the notion of \textit{Seifert circles} of an oriented link diagram. Second, the graph $G'(x)$ defined above is the underlying graph of the link diagram for the torus link $T(N,N)$ drawn as the braid closure of the full twist $(\sigma_1\cdots\sigma_{N-1})^N$ in the braid group $B_N$, with a possibly nonstandard orientation, see \cref{fig:T55}. See e.g. \cite{birman1974braids,yamada1987minimal} for relevant definitions in knot theory. 

\begin{proof}[Proof of \cref{prop:loop-count}(4)]
By Yamada \cite{yamada1987minimal}, the number of Seifert circles in any diagram of an oriented link $L$ is bounded below by its bridge index. On the other hand, the bridge index of a link is bounded below by its number of components. Apply these to the $N$-component link $L=T(N,N)$ with the orientation corresponding to the given pattern $x$, we see the Hitomezashi loop count is bounded below by $N$.
\end{proof}

We proceed to establish some general structural results before proving \cref{lem:switch}. Every link-like graph admits a checkerboard coloring, from which the following lemma is immediate by considering the color on the left of a Hitomezashi loop.
\begin{lemma}\label{lem:vertex_once}
A Hitomezashi loop of a link-like graph passes any vertex at most once.\qed
\end{lemma}

Lastly we remark that the annulus admits a $\ZZ/2$-symmetry by reflecting across and do a half rotation along the central circle (which exchanges the horizontal and vertical sides of each of the $N$ strand respectively). It will be of use that our graph $G'(x)$, thus the whole Hitomezashi pattern, is invariant under this $\ZZ/2$-symmetry.

\subsection{The triple point move}
The triple point move, also known as the Reidemeister III move in the knot theoretic context, is the local move that changes one link-like graph to another as shown in \cref{fig:triple}, thinking as passing the top strand across the bottom intersection point. The orientations of the strands do not change under the move. Note this move is reversible and cyclically symmetric.
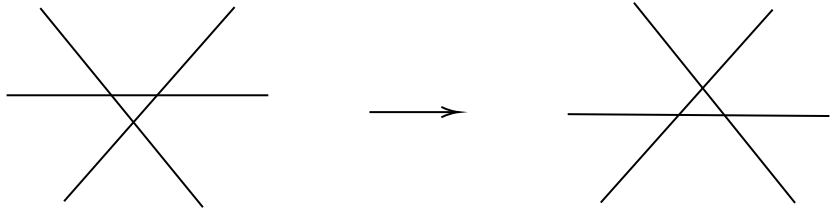
\begin{figure}[t]
    \centering
\tikzset{every picture/.style={line width=0.75pt}} 

\tikzset{every picture/.style={line width=0.75pt}} 

\begin{tikzpicture}[x=0.75pt,y=0.75pt,yscale=-1,xscale=1]

\draw    (109,114) -- (241,114) ;
\draw    (126,70) -- (208,170.5) ;
\draw    (224,69.5) -- (138,167.5) ;
\draw    (292,122.5) -- (335,122.5) ;
\draw [shift={(337,122.5)}, rotate = 180] [color={rgb, 255:red, 0; green, 0; blue, 0 }  ][line width=0.75]    (8.74,-2.63) .. controls (5.56,-1.12) and (2.65,-0.24) .. (0,0) .. controls (2.65,0.24) and (5.56,1.12) .. (8.74,2.63)   ;
\draw    (524,124.47) -- (392,123.53) ;
\draw    (506.68,168.35) -- (425.41,67.27) ;
\draw    (408.68,168.15) -- (495.38,70.77) ;

\end{tikzpicture}
    \caption{The triple point move}
    \label{fig:triple}
\end{figure}

Label the strands of $G'(x)$ from bottom to top by $1,2,\cdots,N$ as in Figure~\ref{fig:10110_move}. Perform a sequence of triple point moves by passing the first strand across the $2N-2$ intersection points of the second strand with the rest strands (marked blue in \cref{fig:10110_move}), bottom to top, left to right, we arrive at a link-like graph isomorphic to $G'(x')$ where $x'$ is the binary string obtained by exchanging the first two bits of $x$.

To prove Lemma~\ref{lem:switch}, we examine the change of the number of Hitomezashi loops, and their homology classes, under a general triple point move.

The three strands involved in a triple point move have either adjacent (\cref{fig:adjacent}) or alternating orientations (\cref{fig:alternating}).

\begin{lemma}[Triple point move for adjacent orientations]\label{lem:adjacent}
If the orientations of the three strands in a triple point move are adjacent, then the number of Hitomezashi loops does not change under the triple point move.
\end{lemma}
\begin{figure}[t]
\centering

\tikzset{every picture/.style={line width=0.75pt}} 

\begin{tikzpicture}[x=0.75pt,y=0.75pt,yscale=-.8,xscale=.8]

\draw    (292,122.5) -- (335,122.5) ;
\draw [shift={(337,122.5)}, rotate = 180] [color={rgb, 255:red, 0; green, 0; blue, 0 }  ][line width=0.75]    (8.74,-2.63) .. controls (5.56,-1.12) and (2.65,-0.24) .. (0,0) .. controls (2.65,0.24) and (5.56,1.12) .. (8.74,2.63)   ;
\draw [color={rgb, 255:red, 74; green, 144; blue, 226 }  ,draw opacity=1 ][line width=2]    (127.06,181) -- (173.39,126.3) ;
\draw [color={rgb, 255:red, 74; green, 144; blue, 226 }  ,draw opacity=1 ][line width=2]    (155.86,104.73) -- (173.39,126.3) ;
\draw [color={rgb, 255:red, 74; green, 144; blue, 226 }  ,draw opacity=1 ][line width=2]    (92,102.42) -- (155.86,104.73) ;
\draw [color={rgb, 255:red, 208; green, 2; blue, 27 }  ,draw opacity=1 ][line width=2]    (155.86,104.73) -- (192.17,104.73) ;
\draw [color={rgb, 255:red, 208; green, 2; blue, 27 }  ,draw opacity=1 ][line width=2]    (173.39,126.3) -- (192.17,104.73) ;
\draw [color={rgb, 255:red, 126; green, 211; blue, 33 }  ,draw opacity=1 ][line width=2]    (192.17,104.73) -- (241,104.73) ;
\draw  [color={rgb, 255:red, 0; green, 0; blue, 0 }  ,draw opacity=1 ][line width=1.5]  (121.03,91.84) -- (105.72,102.41) -- (121.03,112.98) ;
\draw [color={rgb, 255:red, 126; green, 211; blue, 33 }  ,draw opacity=1 ][line width=2]    (192.17,104.73) -- (224.72,64.68) ;
\draw [color={rgb, 255:red, 208; green, 2; blue, 27 }  ,draw opacity=1 ][line width=2]    (173.39,126.3) -- (217.21,177.92) ;
\draw [color={rgb, 255:red, 208; green, 2; blue, 27 }  ,draw opacity=1 ][line width=2]    (118.92,62.75) -- (155.86,104.73) ;
\draw  [color={rgb, 255:red, 0; green, 0; blue, 0 }  ,draw opacity=1 ][line width=1.5]  (147.71,80.59) -- (130.84,75.4) -- (135.26,95.96) ;
\draw  [color={rgb, 255:red, 0; green, 0; blue, 0 }  ,draw opacity=1 ][line width=1.5]  (210.54,98.41) -- (213.72,77.66) -- (197.17,83.78) ;
\draw [color={rgb, 255:red, 74; green, 144; blue, 226 }  ,draw opacity=1 ][line width=2]    (383.22,131.99) -- (447.08,134.3) ;
\draw [color={rgb, 255:red, 208; green, 2; blue, 27 }  ,draw opacity=1 ][line width=2]    (447.08,134.3) -- (483.39,134.3) ;
\draw [color={rgb, 255:red, 208; green, 2; blue, 27 }  ,draw opacity=1 ][line width=2]    (447.08,134.3) -- (465.86,112.73) ;
\draw [color={rgb, 255:red, 126; green, 211; blue, 33 }  ,draw opacity=1 ][line width=2]    (483.39,134.3) -- (532.22,134.3) ;
\draw  [color={rgb, 255:red, 0; green, 0; blue, 0 }  ,draw opacity=1 ][line width=1.5]  (415.03,121.84) -- (399.72,132.41) -- (415.03,142.98) ;
\draw [color={rgb, 255:red, 126; green, 211; blue, 33 }  ,draw opacity=1 ][line width=2]    (466.17,112.73) -- (498.72,72.68) ;
\draw [color={rgb, 255:red, 208; green, 2; blue, 27 }  ,draw opacity=1 ][line width=2]    (428.92,70.75) -- (465.86,112.73) ;
\draw  [color={rgb, 255:red, 0; green, 0; blue, 0 }  ,draw opacity=1 ][line width=1.5]  (457.71,90.59) -- (440.84,85.4) -- (445.26,105.96) ;
\draw  [color={rgb, 255:red, 0; green, 0; blue, 0 }  ,draw opacity=1 ][line width=1.5]  (484.54,106.41) -- (487.72,85.66) -- (471.17,91.78) ;
\draw [color={rgb, 255:red, 74; green, 144; blue, 226 }  ,draw opacity=1 ][line width=2]    (414.52,174.36) -- (447.08,134.3) ;
\draw [color={rgb, 255:red, 126; green, 211; blue, 33 }  ,draw opacity=1 ][line width=2]    (465.86,112.73) -- (483.39,134.3) ;
\draw [color={rgb, 255:red, 208; green, 2; blue, 27 }  ,draw opacity=1 ][line width=2]    (483.39,134.3) -- (520.32,176.29) ;
\end{tikzpicture}
    \caption{The triple point move for adjacent orientations does not change the topology of Hitomezashi loops.}
    \label{fig:adjacent}
\end{figure}
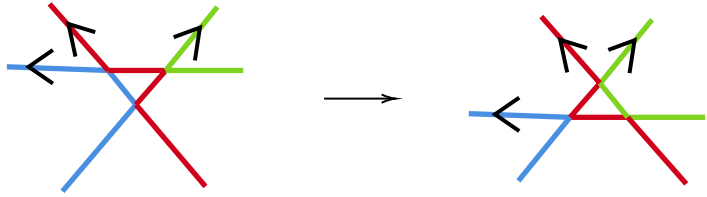
\begin{proof}
The local picture does not change. See \cref{fig:adjacent}. Note by Lemma~\ref{lem:vertex_once} the three colored paths as shown do belong to different Hitomezashi loops, although we do not need this fact for our argument.
\end{proof}

If the orientations of the three strands in a triple point move are alternating, complete the local fragments involved in the move into Hitomezashi loops. Up to overall orientation reversal, cyclic permutation, and passing strands across $\infty$, there are three configurations, as shown in \cref{fig:alternating}. From left to right, we call them Configurations I, II, III, respectively. Then the triple point move switches Configurations I and III and leaves Configuration II invariant. As a corollary we obtain the following.
\begin{lemma}[Triple point move for alternating orientations]\label{lem:alternating}
If the orientations of the three strands in a triple point move are alternating, then under the move, the number of Hitomezashi loops decreases/increases by two for Configurations I, III, and does not change for Configuration II.\qed
\end{lemma}

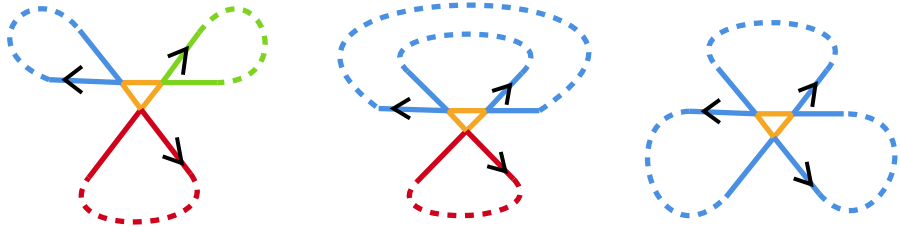
\begin{figure}[t]
    \centering
\tikzset{every picture/.style={line width=0.75pt}} 

\begin{tikzpicture}[x=0.75pt,y=0.75pt,yscale=-.8,xscale=.8]

\draw [color={rgb, 255:red, 208; green, 2; blue, 27 }  ,draw opacity=1 ][line width=2]    (106.8,153.97) -- (139.74,110.99) ;
\draw [color={rgb, 255:red, 245; green, 166; blue, 35 }  ,draw opacity=1 ][line width=2]    (127.28,94.04) -- (139.74,110.99) ;
\draw [color={rgb, 255:red, 74; green, 144; blue, 226 }  ,draw opacity=1 ][line width=2]    (81.88,92.23) -- (127.28,94.04) ;
\draw [color={rgb, 255:red, 245; green, 166; blue, 35 }  ,draw opacity=1 ][line width=2]    (127.28,94.04) -- (153.09,94.04) ;
\draw [color={rgb, 255:red, 245; green, 166; blue, 35 }  ,draw opacity=1 ][line width=2]    (139.74,110.99) -- (153.09,94.04) ;
\draw [color={rgb, 255:red, 126; green, 211; blue, 33 }  ,draw opacity=1 ][line width=2]    (153.09,94.04) -- (187.81,94.04) ;
\draw  [color={rgb, 255:red, 0; green, 0; blue, 0 }  ,draw opacity=1 ][line width=1.5]  (102.52,83.9) -- (91.63,92.21) -- (102.52,100.52) ;
\draw [color={rgb, 255:red, 126; green, 211; blue, 33 }  ,draw opacity=1 ][line width=2]    (153.09,94.04) -- (176.24,62.56) ;
\draw [color={rgb, 255:red, 208; green, 2; blue, 27 }  ,draw opacity=1 ][line width=2]    (139.74,110.99) -- (170.9,151.55) ;
\draw [color={rgb, 255:red, 74; green, 144; blue, 226 }  ,draw opacity=1 ][line width=2]    (101.02,61.05) -- (127.28,94.04) ;
\draw  [color={rgb, 255:red, 0; green, 0; blue, 0 }  ,draw opacity=1 ][line width=1.5]  (153.4,140.59) -- (165.36,144.8) -- (162.37,128.61) ;
\draw  [color={rgb, 255:red, 0; green, 0; blue, 0 }  ,draw opacity=1 ][line width=1.5]  (166.16,89.08) -- (168.41,72.76) -- (156.65,77.58) ;
\draw [color={rgb, 255:red, 208; green, 2; blue, 27 }  ,draw opacity=1 ][line width=2]    (313.39,156.9) -- (344.91,124.53) ;
\draw [color={rgb, 255:red, 245; green, 166; blue, 35 }  ,draw opacity=1 ][line width=2]    (332.98,111.77) -- (344.91,124.53) ;
\draw [color={rgb, 255:red, 74; green, 144; blue, 226 }  ,draw opacity=1 ][line width=2]    (289.54,110.4) -- (332.98,111.77) ;
\draw [color={rgb, 255:red, 245; green, 166; blue, 35 }  ,draw opacity=1 ][line width=2]    (332.98,111.77) -- (357.68,111.77) ;
\draw [color={rgb, 255:red, 245; green, 166; blue, 35 }  ,draw opacity=1 ][line width=2]    (344.91,124.53) -- (357.68,111.77) ;
\draw [color={rgb, 255:red, 74; green, 144; blue, 226 }  ,draw opacity=1 ][line width=2]    (357.68,111.77) -- (390.91,111.77) ;
\draw  [color={rgb, 255:red, 0; green, 0; blue, 0 }  ,draw opacity=1 ][line width=1.5]  (309.29,104.14) -- (298.87,110.39) -- (309.29,116.65) ;
\draw [color={rgb, 255:red, 74; green, 144; blue, 226 }  ,draw opacity=1 ][line width=2]    (357.68,111.77) -- (379.83,88.06) ;
\draw [color={rgb, 255:red, 208; green, 2; blue, 27 }  ,draw opacity=1 ][line width=2]    (344.91,124.53) -- (374.72,155.07) ;
\draw [color={rgb, 255:red, 74; green, 144; blue, 226 }  ,draw opacity=1 ][line width=2]    (307.85,86.92) -- (332.98,111.77) ;
\draw  [color={rgb, 255:red, 0; green, 0; blue, 0 }  ,draw opacity=1 ][line width=1.5]  (357.98,146.82) -- (369.42,149.99) -- (366.56,137.8) ;
\draw  [color={rgb, 255:red, 0; green, 0; blue, 0 }  ,draw opacity=1 ][line width=1.5]  (370.18,108.03) -- (372.34,95.75) -- (361.09,99.37) ;
\draw [color={rgb, 255:red, 74; green, 144; blue, 226 }  ,draw opacity=1 ][line width=2]    (508.54,166.27) -- (538.71,128.6) ;
\draw [color={rgb, 255:red, 245; green, 166; blue, 35 }  ,draw opacity=1 ][line width=2]    (527.29,113.74) -- (538.71,128.6) ;
\draw [color={rgb, 255:red, 74; green, 144; blue, 226 }  ,draw opacity=1 ][line width=2]    (485.71,112.14) -- (527.29,113.74) ;
\draw [color={rgb, 255:red, 245; green, 166; blue, 35 }  ,draw opacity=1 ][line width=2]    (527.29,113.74) -- (550.94,113.74) ;
\draw [color={rgb, 255:red, 245; green, 166; blue, 35 }  ,draw opacity=1 ][line width=2]    (538.71,128.6) -- (550.94,113.74) ;
\draw [color={rgb, 255:red, 74; green, 144; blue, 226 }  ,draw opacity=1 ][line width=2]    (550.94,113.74) -- (582.74,113.74) ;
\draw  [color={rgb, 255:red, 0; green, 0; blue, 0 }  ,draw opacity=1 ][line width=1.5]  (504.62,104.85) -- (494.64,112.13) -- (504.62,119.42) ;
\draw [color={rgb, 255:red, 74; green, 144; blue, 226 }  ,draw opacity=1 ][line width=2]    (550.94,113.74) -- (572.14,86.14) ;
\draw [color={rgb, 255:red, 74; green, 144; blue, 226 }  ,draw opacity=1 ][line width=2]    (538.71,128.6) -- (567.24,164.15) ;
\draw [color={rgb, 255:red, 74; green, 144; blue, 226 }  ,draw opacity=1 ][line width=2]    (503.24,84.81) -- (527.29,113.74) ;
\draw  [color={rgb, 255:red, 0; green, 0; blue, 0 }  ,draw opacity=1 ][line width=1.5]  (551.22,154.54) -- (562.17,158.23) -- (559.43,144.04) ;
\draw  [color={rgb, 255:red, 0; green, 0; blue, 0 }  ,draw opacity=1 ][line width=1.5]  (562.9,109.38) -- (564.97,95.08) -- (554.2,99.3) ;
\draw [color={rgb, 255:red, 74; green, 144; blue, 226 }  ,draw opacity=1 ][line width=2]  [dash pattern={on 3.38pt off 3.27pt}]  (81.88,92.23) .. controls (29.26,71.46) and (71.92,24.31) .. (101.02,61.05) ;
\draw [color={rgb, 255:red, 208; green, 2; blue, 27 }  ,draw opacity=1 ][line width=2]  [dash pattern={on 3.38pt off 3.27pt}]  (170.9,151.55) .. controls (199.19,190.12) and (76.9,193.27) .. (106.8,153.97) ;
\draw [color={rgb, 255:red, 126; green, 211; blue, 33 }  ,draw opacity=1 ][line width=2]  [dash pattern={on 3.38pt off 3.27pt}]  (176.24,62.56) .. controls (211.99,15.66) and (243.27,93.46) .. (187.81,94.04) ;
\draw [color={rgb, 255:red, 208; green, 2; blue, 27 }  ,draw opacity=1 ][line width=2]  [dash pattern={on 3.38pt off 3.27pt}]  (374.72,155.07) .. controls (401.79,184.12) and (284.78,186.48) .. (313.39,156.9) ;
\draw [color={rgb, 255:red, 74; green, 144; blue, 226 }  ,draw opacity=1 ][line width=2]  [dash pattern={on 3.38pt off 3.27pt}]  (379.83,88.06) .. controls (414.72,53.94) and (275.25,56.89) .. (307.85,86.92) ;
\draw [color={rgb, 255:red, 74; green, 144; blue, 226 }  ,draw opacity=1 ][line width=2]  [dash pattern={on 3.38pt off 3.27pt}]  (390.91,111.77) .. controls (525.61,22.57) and (171.85,23.76) .. (289.54,110.4) ;
\draw [color={rgb, 255:red, 74; green, 144; blue, 226 }  ,draw opacity=1 ][line width=2]  [dash pattern={on 3.38pt off 3.27pt}]  (572.14,86.14) .. controls (605.53,46.41) and (468.78,46.41) .. (503.24,84.81) ;
\draw [color={rgb, 255:red, 74; green, 144; blue, 226 }  ,draw opacity=1 ][line width=2]  [dash pattern={on 3.38pt off 3.27pt}]  (567.24,164.15) .. controls (602.92,205.54) and (645.9,113.23) .. (582.74,113.74) ;
\draw [color={rgb, 255:red, 74; green, 144; blue, 226 }  ,draw opacity=1 ][line width=2]  [dash pattern={on 3.38pt off 3.27pt}]  (485.71,112.14) .. controls (435.57,113.23) and (463.57,211.05) .. (508.54,166.27) ;

\end{tikzpicture}
    \caption{Configurations I, II, III for the triple point move of alternating orientations.}
    \label{fig:alternating}
\end{figure}

\subsection{Proof of \texorpdfstring{\cref{lem:switch}}{Lemma 4.2}}
We perform the sequence of triple point move switching the first two strands, as described in the previous section. More carefully, we pass the first strand alternatingly across the vertical array (from bottom to top) and the horizontal array (from left to right) of the intersection points (as marked blue in \cref{fig:10110_move}). For example, in \cref{fig:10110_move}, the first pair of moves would pass strand $1$ across the two intersections between strand $2$ and strand $3$. Now \cref{lem:switch} is implied by the following claim.

\textbf{Claim:} The number of Hitomezashi loops modulo $4$ does not change under every pair of triple point moves (i.e. the two moves that pass the first strand across the two intersections between the second strand with one of the rest strands) performed.

\begin{proof}[Proof of claim]
Note that the $\ZZ/2$-symmetry mentioned at the end of \cref{sec:annulus} is preserved under each pair of moves. Therefore, the two sets of orientations for a pair of triple point moves are both adjacent or both alternating. If they are both adjacent, we are done by Lemma~\ref{lem:adjacent}.

Suppose now they are both alternating. Then the two triple point moves have the same configuration type as shown in \cref{fig:alternating}. Each of the two local pictures has six endpoints, and these twelve endpoints are paired up by six Hitomezashi paths (see \cref{fig:six_pairs}). By the $\ZZ/2$-symmetry, among these six, there must be an even number of \textit{cross} Hitomezashi paths, i.e. paths that pair up two endpoints from different local pictures. We divide into three cases.

\textbf{Case 1} There is no cross Hitomezashi path.

In this case, performing one of the triple point moves does not change the configuration type of the other move. Therefore the number of Hitomezashi loops changes by $\pm4$ or $0$ under the pair of moves by \cref{lem:alternating}.

\textbf{Case 2} There are exactly two cross Hitomezashi paths.

Then in each local picture, four of the six endpoints are paired up by non-cross paths. With this observation, we conclude verbatim as in Case 1.

\textbf{Case 3} There are at least four cross Hitomezashi paths.

\input{Figures/six_pairs}

Label six endpoints by $1,2,\cdots,6$ clockwisely in one local picture, and the other six also by $1,2,\cdots,6$ such that the labels are $\ZZ/2$-invariant. By purely topological consideration, if there are cross paths between $1$'s and $2$'s or $1$'s and $6$'s, then there must be exactly two cross paths, which contradicts our assumption. Therefore if the $1$'s are on cross paths, there must be two cross paths between $1$'s and $4$'s. Cyclic permutations of this conclusion also hold. 

If two of the $(1,4)$, $(2,5)$, $(3,6)$ endpoints are paired by cross paths, so is the third. Otherwise, a non-cross path between $3$ and $6$ in one local picture together with a path within this local picture form a loop that divides the annulus into two parts, one of which contains the entire other local picture and half of the current local picture. The other half contains the other half of the current local picture, on which the two vertices cannot be endpoints of cross paths.

Therefore there is exactly one possible configuration (at least when forgetting about how the loops embed into the annulus) as shown on the left of \cref{fig:six_pairs}. In this case, performing the pair of triple point moves does not change the number of Hitomezashi loops.
\end{proof}

\section{Open problems}
\label{sec:problems}
We point out some potentially interesting directions for future study.
\begin{problem}
When $k(x) = k(y) = 0$, we know from \cref{prop:homology} that the possible homology classes of nontrivial Hitomezashi loops in $\Cloth_{M, N}(x, y)$ are $(\pm 1, 0)$ and $(0, \pm 1)$. However, by \cref{obs}(2), loops with homology class $(\pm 1, 0)$ and $(0, \pm 1)$ cannot coexist in the same pattern. Is there a simple criterion on $x, y$ that tells us which homology classes can exist in this case?
\end{problem}
\begin{problem}
We could not formulate a generalization of \cref{prop:loop-count}(2) for non-symmetric Hitomezashi patterns. In particular, the total number of loops modulo $4$ is not a function of $M, N, k(x), k(y)$, as illustrated in \cref{fig:loop_count}. This is still the case if we assume $k(x) \neq 0, k(y) \neq 0$, as the number of loops is congruent to $0$ modulo $4$ in $\Cloth_{8,8}(x,x)$ and $2$ modulo $4$ in $\Cloth_{8,8}(x, y)$ for $x = +-++--++$ and $y = +-++-+-+$. 

Could we find a formula for the number of loops modulo $4$ in a general Hitomezashi pattern $\Cloth_{M, N}(x, y)$?
\end{problem}

\begin{figure}[t]
    \centering
    \begin{minipage}{0.5\textwidth}
        \centering
        \includegraphics[width=0.7\textwidth]{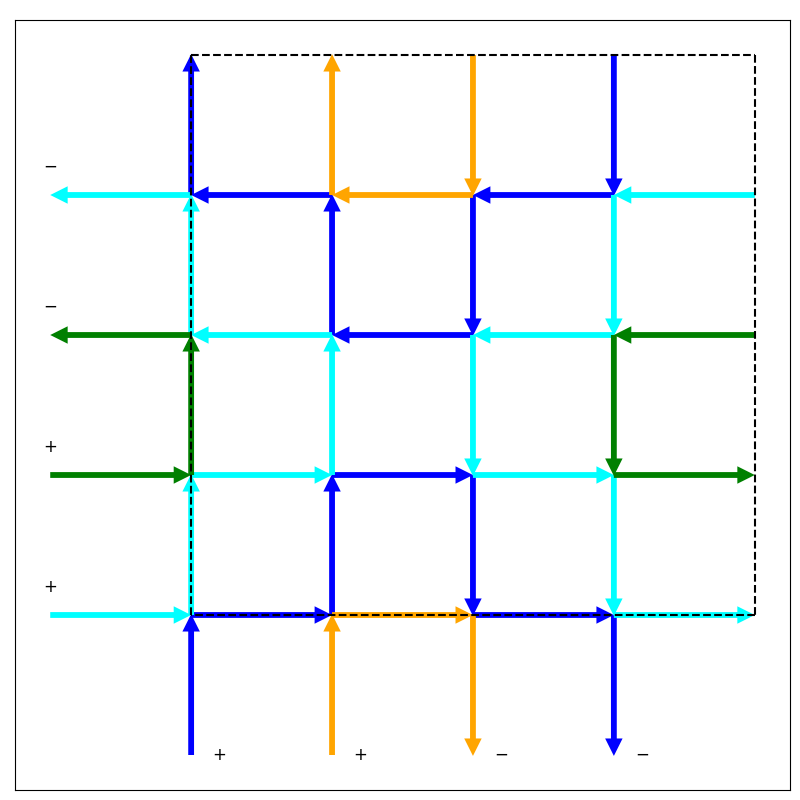} 
    \end{minipage}\hfill
    \begin{minipage}{0.5\textwidth}
        \centering
        \includegraphics[width=0.7\textwidth]{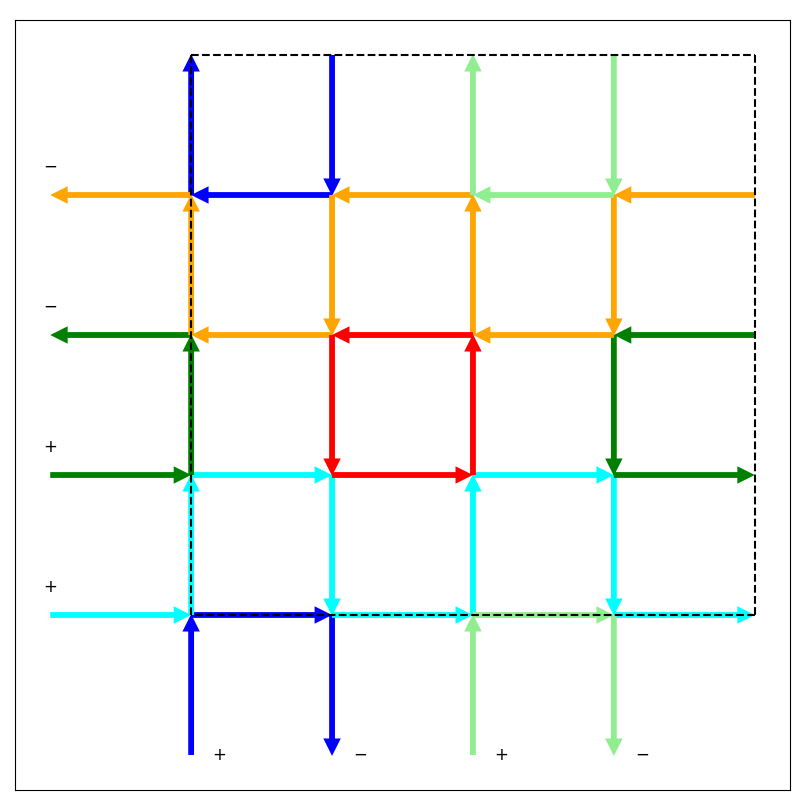} 
    \end{minipage}
    \caption{When $x = ++--$ and $x' = +-+-$, $\Cloth_{4,4}(x, x)$ has $4$ toroidal Hitomezashi loops, while $\Cloth_{4,4}(x, x')$ has $6$.}
    \label{fig:loop_count}
\end{figure}

\begin{problem}
As explained in the introduction, in the case when $M,N$ are both even, for a fixed Hitomezashi pattern $\Cloth_{M,N}(x,y)$, the set of Hitomezashi loops decompose into two parts that are dual to each other. What can one say about the loop count of each part?
\end{problem}
\begin{problem}
Generalize our results to Hitomezashi patterns on an infinite cylindrical grid $\ZZ\times\ZZ/N\ZZ$.
\end{problem}
\begin{problem}
Colin, Defant, and Tenner \cite{defant2023} sketched some reasonable definitions for Hitomezashi patterns on the standard planar triangular grid whose vertex set is given by $T=\{i+j\omega+k\omega^2\colon i,j,k\in\ZZ\}\subset\CC$, where $\omega=e^{2\pi i/3}$. Can one find a reformulation of their definitions in the spirit of Definition~\ref{defn:hitomezashi-toroidal}? What can be said about planar triangular Hitomezashi patterns and toroidal triangular Hitomezashi patterns? A reasonable toroidal triangular grid to consider is $T/NT$ for some integer $N\ge2$.
\end{problem}

\bibliographystyle{amsplain}
\bibliography{bib}

\end{document}